\documentclass[reqno]{amsart}
\usepackage{geometry} 
\usepackage{amsmath,amscd,amssymb,amscd,amsthm,amsfonts,graphicx,mathrsfs}
\usepackage[foot]{amsaddr}

\newcommand{\sA}{\mathcal{A}}
\newcommand{\ta}{\tilde{a}}
\newcommand{\tA}{\tilde{A}}
\newcommand{\tb}{\tilde{b}}
\newcommand{\tB}{\tilde{B}}
\newcommand{\de}{\delta}
\newcommand{\De}{\Delta}
\newcommand{\tDP}{\tilde{DP}}
\newcommand{\e}{\epsilon}
\newcommand{\sF}{\mathcal{F}}
\newcommand{\g}{\gamma}
\newcommand{\dg}{\dot{\gamma}}

\newcommand{\hg}{\hat{g}}
\newcommand{\tg}{\tilde{g}}
\newcommand{\sG}{\mathcal{G}}
\newcommand{\J}{\mathbb{J}}
\newcommand{\hJ}{\hat{J}}

\newcommand{\Jd}{\Delta(t,x)}
\newcommand{\tk}{\tilde{k}}
\newcommand{\hK}{\hat{K}}
\newcommand{\sK}{\mathcal{K}}

\newcommand{\la}{\lambda}

\newcommand{\vphi}{\varphi}
\newcommand{\R}{\mathbb{R}}
\newcommand{\sR}{\mathcal{R}}
\newcommand{\tR}{\tilde{R}}
\newcommand{\fS}{\mathfrak{S}}
\newcommand{\sS}{\mathcal{S}}
\newcommand{\si}{\sigma}
\newcommand{\Si}{\Sigma}
\newcommand{\sU}{\mathcal{U}}
\newcommand{\tU}{\tilde{U}}
\newcommand{\del}{\partial}
\newcommand{\supp}{\text{supp}\,}
\newcommand{\Id}{\text{Id}}
\newcommand{\Mat}{\text{Mat}}
\newcommand{\Norm}[1]{\left\lVert#1\right\rVert}
\newcommand{\norm}[1]{\lVert#1\rVert}
\newcommand{\ds}{\displaystyle}
\newcommand{\nin}{\noindent}

\newtheorem{theorem}{Theorem}
\newtheorem*{theorem1}{Theorem 1}
\newtheorem*{theorem2}{Theorem 2}
\newtheorem{lemma}[theorem]{Lemma}
\newtheorem{corollary}[theorem]{Corollary}

\def\@setaddresses{\par
\nobreak \begingroup
\def\author##1{\nobreak\addvspace\bigskipamount}%
\def\\{\unskip, \ignorespaces}%
\interlinepenalty\@M
\def\address##1##2{}%
\def\email##1##2{}%
\def\curraddr##1##2{\begingroup
\@ifnotempty{##2}{\nobreak\indent{\itshape Current address}%
\@ifnotempty{##1}{, \ignorespaces##1\unskip}\/:\space
##2\par}\endgroup}%
\def\urladdr##1##2{\begingroup
\@ifnotempty{##2}{\nobreak\indent{\itshape URL}%
\@ifnotempty{##1}{, \ignorespaces##1\unskip}\/:\space
\ttfamily##2\par}\endgroup}%
\addresses
\endgroup
}

\begin{document}

\title[Franks' lemma for geodesic flows]{An elementary proof of Franks' lemma for geodesic flows}
\thanks{This material is based upon parts of the author's Ph.D. thesis, as well as work supported by the National Science Foundation under Grant Number NSF 1045119.}

\author{Daniel Visscher}

\address{Department of Mathematics, University of Michigan, Ann Arbor, MI, USA }
%  \email{davissch@umich.edu} 

\date{}


\begin{abstract}
Given a Riemannian manifold $(M,g)$ and a geodesic $\gamma$, the perpendicular part of the derivative of the geodesic flow $\phi_g^t: SM \rightarrow SM$ along $\gamma$ is a linear symplectic map. We give an elementary proof of the following Franks' lemma, originally found in \cite{CP02} and \cite{Con10}: this map can be perturbed freely within a neighborhood in $Sp(n)$ by a $C^2$-small perturbation of the metric $g$ that keeps $\gamma$ a geodesic for the new metric. Moreover, the size of these perturbations is uniform over fixed length geodesics on the manifold. When $\dim M \geq 3$, the original metric must belong to a $C^2$--open and dense subset of metrics.
\end{abstract}

\maketitle

\section{Introduction}
\label{intro}

The derivative of a diffeomorphism or flow along an orbit carries substantial information about the dynamics along that orbit. For instance, suppose $x$ is a periodic point for a diffeomorphism $f$ of period $n$. If the eigenvalues of $D_xf^n$ have modulus bounded away from $1$, then the orbit is called hyperbolic. In this case, provided the eigenvalues have non-zero real part, the Hartman--Grobman theorem states that $f^n$ is topologically conjugate to its linearization $D_xf^n$ in a neighborhood of $x$. That is, all topological dynamical information is contained in the derivative at the hyperbolic fixed point. One question this type of analysis raises is how the linearization along an orbit depends on the dynamical system.

A Franks' lemma is a tool that allows one to freely perturb the derivative of a diffeomorphism or flow along a finite piece of orbit. The name alludes to a lemma proved by John Franks for diffeomorphisms in \cite{Fr71}, in which the desired linear maps along an orbit are pasted in via the exponential map. This type of result is important in the study of stable properties of a dynamical system, since it allows one to equate stablility over perturbations in a space of diffeomorphisms or flows with stability over perturbations in a linear space. Franks' lemmas have since been proven and used in other contexts; see, for instance, \cite{BDP03} for conservative diffeomorphisms, \cite{Arn02}, \cite{HT06} and \cite{ALD13} for symplectomorphisms, \cite{MPP04} and \cite{BGV06} for flows, \cite{BR08} for conservative flows, \cite{Viv05} for Hamiltonians, and \cite{CP02} and \cite{Con10} for geodesic flows. A priori, more restricted settings are more difficult to work with. For instance, using a Franks' lemma for Hamiltonians, one can perturb the derivative of the geodesic flow along an orbit by perturbing the generating Hamiltonian function, but the new Hamiltonian flow may not be a {\em geodesic} flow (coming from a Riemannian metric on a manifold).

This paper provides an elementary proof of the Franks' lemma for geodesic flows on surfaces found in \cite{CP02}, and its higher--dimensional analogue as found in \cite{Con10}. In the latter, the author notes that this Franks' lemma is ``the main technical difficulty in the paper.'' One aim of the present paper is to make this result more intuitive.

%%%

Let $(M,g)$ be a closed Riemannian manifold, $SM$ the sphere bundle\footnote{The geodesic flow is usually defined on the unit tangent bundle, but this space is not preserved under perturbations of the metric. It is clear that the sphere bundle can be naturally identified with the unit tangent bundle for any metric, however, and we will make this identification when talking about $SM$.} over $M$, and $\phi_g: \R \times SM \rightarrow SM$ the geodesic flow for the metric $g$. Since $M$ is compact, there is some length $\ell>0$ for which any geodesic segment of length less than $\ell$ has no self--intersections. We will assume for convenience that $\ell=1$ (this can be achieved by scaling). 

Consider, then, a length--$1$ geodesic $\g$, and its path $\dg \subset SM$. Pick local hypersurfaces $\Si_0$ and $\Si_1$ in $SM$ that are transverse to $\dg(t)$ at $t=0$ and $t=1$, respectively. This allows us to define a Poincar\'e map $P: \Si_0 \supset U \rightarrow \Si_1$, where $U$ is a neighborhood of $\dg(0)$, taking $\xi \in U$ to $\phi_g^{t_1}(\xi)$, where $t_1$ is the smallest positive time such that $\phi_g^{t_1}(\xi) \in \Si_1$. One can use the Implicit Function Theorem and the fact that $\phi_g^t$ is differentiable to show that $P$ is differentiable, with derivative $DP: T_{\dg(0)}\Si_0 \to T_{\dg(1)} \Si_1$.

The map $DP$ contains information about the dynamics along $\g$. If $\g$ is closed (i.e. $\dg(T)=\dg(0)$ for some $T>0$), and $\Si_1=\Si_0$, then the eigenvalues of $DP$ determine whether $\g$ is hyperbolic, elliptic, degenerate, or otherwise. For a closed orbit, this is an invariant of the section $\Si_0$ chosen, since a different section yields a new map that is linearly conjugate to the original and therefore has the same eigenvalues. For non-closed geodesic segments as above, we will consider only hypersurfaces orthogonal to  $\dg$, which allows us make sense of $DP$ as a linear symplectic map (since $\phi_g^t$ preserves a symplectic form). Moreover, putting coordinates on a neighborhood of $\g$ will allow us to write $DP$ as an element of $Sp(n) = \{ A \in Mat_{2n}(\R) | A^T \J A = \J \}$, where $$\J = \left[ \begin{array}{cc}
0 & I_n \\
-I_n & 0 \end{array} \right].$$

We first consider a Franks' lemma for geodesic flows on surfaces, where the proof techniques are particularly simple and apply to any metric. Let $M$ be a compact manifold, and $\sG^r(M)$ the set of $C^r$ metrics on $M$ equipped with the $C^r$ topology. For a given path $\g$ on $M$, let $\sG_\g^r(M)$ be the set of $C^r$ metrics on $M$ for which $\g$ is a geodesic. The following theorem (``Franks' lemma for geodesic flows on surfaces'') states that on any surface $(S,g)$, the linear map $DP$ along any length--$1$ geodesic segment can be freely perturbed in a neighborhood inside $Sp(1)$ by a $C^2$--small perturbation of the metric.

\begin{theorem}[\cite{CP02}]\label{FLGF-S}
Let $g \in \sG^4(S)$ and let $\sU$ be a neighborhood of $g$ in $\sG^2(S)$.  Then there exists $\de = \de(g,\sU) >0$ such that for any simple geodesic segment $\g$ of length $1$, each element of $B(DP(\g,g),\de) \subset Sp(1)$ is realizable as $DP(\g, \tg)$ for some $\tg \in \sU \cap \sG_\g(S)$.  Moreover, for any tubular neighborhood $W$ of $\g$ and any finite set $\sF$ of transverse geodesics, the support of the perturbation can be contained in $W \setminus V$ for some small neighborhood $V$ of the transverse geodesics $\sF$.
\end{theorem}

A couple of notes on the statement of the theorem. Distance in $Sp(1)$ comes from using coordinates to identify it with $\R^3$ (and the Euclidean norm on $\R^3$), while the $C^2$ distance in the space of metrics comes from fixing a coordinate system and using the $C^2$ norm on the metric matrix component functions. Theorem~\ref{FLGF-S} states that $\de$ can be chosen uniformly over all geodesics of length $1$---it depends only on the metric $g$ (more specifically, the bounds on its curvature), and the neighborhood $\sU$. As noted in \cite{CP02}, this can be relaxed to geodesics of length $\ell$ in some interval $[a,b]$, but then $\de = \de(g,\sU,a,b)$ depends on the upper and lower bounds of the length. The assumption that the original metric is $C^4$ is used to imply that the curvature $K$ is $C^2$, which is needed for the estimates of Lemma \ref{comp}.

The proof techniques for Theorem \ref{FLGF-S} can be generalized to higher dimensions, but not for every metric. In particular, we need the dynamics of the geodesic flow along $\g$ to do some of the work for us, since the higher--dimensional analogues of the perturbations that we use to prove Theorem \ref{FLGF-S} do not produce a full dimensional ball in $Sp(n)$. The suitable metrics constitute the set $\sG_1$, which is proven to be $C^2$ open and $C^\infty$ dense in $\sG^2(M)$ in \cite{Con10}.

\begin{theorem}[\cite{Con10}]\label{FLGF-HD}
Let $g \in \sG^4(M) \cap \sG_1$ and let $\sU$ be a neighborhood of $g$ in $\sG^2(M)$.  Then there exists $\de = \de(g,\sU) >0$ such that for any simple geodesic segment $\g$ of length $1$, each element of $B(DP(\g,g),\de) \subset Sp(n)$ is realizable as $DP(\g, \tg)$ for some $\tg \in \sU \cap \sG_\g(M)$.  Moreover, for any tubular neighborhood $W$ of $\g$ and any finite set $\sF$ of transverse geodesics, the support of the perturbation can be contained in $W \setminus V$ for some small neighborhood $V$ of the transverse geodesics $\sF$.
\end{theorem}

F. Klok proved a similar result in \cite{Klo83}, where he is in fact able to perturb the $k$-jet of the Poincar\'e map along any geodesic in any direction for a $C^{k+1}$--open and dense set of metrics. This result does not contain the uniformity of the size of the perturbation over the choice of geodesic, however, which is a necessary component for \cite{CP02} and \cite{Con10}.

This result can be applied to segments along any finite--length geodesic (e.g. closed geodesics) to assemble a perturbation over the whole geodesic. Such a geodesic may intersect itself many times on the manifold, and in order to keep the curve $\g$ a geodesic in the new metric, one should avoid changing the metric at the intersection points. In this case, the second statement in Theorems \ref{FLGF-S} and \ref{FLGF-HD} regarding the support of these perturbations is important in order to assemble them along a closed (or finite length) geodesic, as done in \cite{CP02} and \cite{Con10}, to yield Corollary \ref{FLGF-S-Cor}. Since a closed geodesic may have non--integer length, we recall that the above theorems can be applied to geodesics of length $\ell \in [a,b]$, with $\de = \de(g,\sU,a,b)$ depending on the upper and lower bounds of the length.

\begin{corollary}[\cite{Con10},\cite{CP02}]\label{FLGF-S-Cor}
Let $g$ be as in Theorem \ref{FLGF-S} or \ref{FLGF-HD}, and let $\sU$ be a neighborhood of $g$ in $\sG^2(M)$. Then there exists $\de = \de(g,\sU) >0$ such that for any prime closed geodesic $\g$, there is an integer $m=m(\g)>0$ such that $\g$ is the concatanation of segments $\g_1, \ldots \g_m$ and any element of the product of the balls of radius $\de$ about $DP(\g_i, g)$ in $Sp(n)^m$ is realizable as $\prod_{i=1}^m DP(\g_i, \tg)$ for some $\tg \in \sU$.
\end{corollary}

Note that now the uniformity of the perturbation shows up as the \emph{radius} of a ball in $Sp(n)^m$, whose \emph{volume} decreases as $m$ grows (for $\de<1$). That is, the size of the perturbation along the geodesic $\g$ is proportional to the number of pieces it must be cut up into to apply Theorem \ref{FLGF-S} or \ref{FLGF-HD} (along with, as above, $g$, $\sU$, and the length of the pieces of geodesic).

%%%%%%%%%%%%%%%%%%%%%
%  FRANKS' LEMMA FOR GEODESIC FLOWS ON SURFACES
%%%%%%%%%%%%%%%%%%%%%
\medskip
\section{Proof of Franks' lemma for geodesic flows on surfaces}
\label{FLGF-SS}
%%%%%%%

This section contains a proof of Theorem \ref{FLGF-S}, using Jacobi fields as the intermediary between the dynamics along $\g$ (i.e. the map $DP$) and the Riemannian metric on a surface $S$.

\begin{theorem1}[\cite{CP02}]
Let $g \in \sG^4(S)$ and let $\sU$ be a neighborhood of $g$ in $\sG^2(S)$.  Then there exists $\de = \de(g,\sU) >0$ such that for any simple geodesic segment $\g$ of length $1$, each element of $B(DP(\g,g),\de) \subset Sp(1)$ is realizable as $DP(\g, \tg)$ for some $\tg \in \sU \cap \sG_\g(S)$.  Moreover, for any tubular neighborhood $W$ of $\g$ and any finite set $\sF$ of transverse geodesics, the support of the perturbation can be contained in $W \setminus V$ for some small neighborhood $V$ of the transverse geodesics $\sF$.
\end{theorem1}

The proof is organized as follows. We construct three curves of metrics in $\sG^2(S)$ passing through $g$ with the property that the images of these curves in $Sp(1)$ under the map $DP(\g, \cdot): \sG^2(S) \rightarrow Sp(1)$ span the tangent space at $DP(\g,g)$. Then the Inverse Function Theorem provides the desired open ball. In Section \ref{pertDP}, we use a relation between the map $DP=DP(\g,g)$ and Jacobi fields along $\g$ to effect a desired perturbation to $DP \in Sp(1)$ by a $C^0$-small perturbation of the curvature $k$ along $\g$. In Section \ref{createmetric}, we build a metric $\tg$ that has the perturbed curvature along $\g$ with the perturbation supported in an arbitrarily small tubular neighborhood of $\g$, and show that $\tg$ is $C^2$-close to the original metric $g$. Then, in Section \ref{avoid}, we show that we can avoid perturbing the metric in a small neighborhood of a finite set of transverse geodesics by switching off and on the perturbation of $k$ very quickly, having a very small effect on Jacobi field values (negligable compared to the size of perturbations). 

Assume that for every geodesic segment of length $1$ on $(S,g)$, the Jacobi field $a$ defined by $a(0)=1$, $a'(0)=0$ is uniformly bounded away from zero on $\g$ (the uniformity is over all such geodesic segments); this can be done because the curvature $K$ takes a maximum value on the compact surface $S$. Further, assume that the injectivity radius of $S$ is at least $2$, so that every geodesic segment of length $1$ will necessarily be non-self-intersecting.  All of this can be achieved by scaling.

A special set of coordinates is well-adapted to studying the dynamics of $\phi_g^t$ along a fixed geodesic, which we will define here more generally on an $(n+1)$--dimensional Riemannian manifold $(M,g)$. Given a non self-intersecting geodesic $\g$ of finite length, say from $\g(0)$ to $\g(1)$, define a set of Fermi coordinates in a tubular neighborhood of $\g$ as follows.  At $t=0$, choose a set of $n$ vectors $\{ e^i \}$ so that $\{ \dot{\g}, e^1, \ldots, e^n \}$ is an orthonormal basis at $T_{\g(0)}M$, and parallel transport this basis to create a frame $\{ \dot{\g}(t), e^1(t), \ldots, e^n(t) \}$ along $\g$.  Exponentiating this frame onto the manifold yields a map $\Phi: [0,1] \times \R^n \rightarrow M$ given by 
$$\Phi(t;{\bf x}) = \exp_{\g(t)} \left[ \sum_{i=1}^n x_i e^i(t) \right].$$  
Since this map has full rank at $\Phi(t; 0)$, it is a diffeomorphism onto a neighborhood $U$ of~$\g$, and so defines coordinates $(t; x_1, \ldots, x_n)$ on $U$.

%%%%%
\subsection{Perturbing $DP$ by perturbing $k$}
\label{pertDP}

On the surface $S$, fix a set of Fermi coordinates $\{ (t,x) \}$ along $\g$, with coordinate neighborhood $U$. A normal Jacobi field along $\g(t) = (t,0)$ is a multiple of $\frac{\del}{\del x}$, and so can be written $J(t) \frac{\del}{\del x}$ with $J(t)$ a scalar. The map $DP=DP(\g, g)$ takes the following form on a normal Jacobi field $J$ along $\g$ (\cite{Pat99}):
$$DP(J(0),J'(0)) = D_{\dg(0)}\phi_g^1(J(0),J'(0)) = (J(1),J'(1)).$$
We will write the pair $(J(t),J'(t))$ as a column vector below, which makes $DP$ into a $2\times2$ matrix. Let $a$ be the Jacobi field defined by $a(0)=1$, $a'(0)=0$ and $b$ the Jacobi field defined by $b(0)=0$, $b'(0)=1$. Then
$$DP = \left[\begin{array}{cc} a(1) & b(1) \\ a'(1) & b'(1)\end{array}\right].$$
Since $DP \in Sp(1)$ and $\dim Sp(1)=3$, one can write $b'(1)$ in terms of $a(1)$, $a'(1)$, and $b(1)$, we will be concerned only with perturbing these latter three values.

Assume that the length of $\g$ is short enough that $a$ is uniformly (over all such $\g$) bounded away from $0$ by $a_{lb}>0$.  Given a nondegenerate solution of a second order ordinary differential equation, a standard reduction of order procudure allows one to write down any other solution in terms of the first. Applied to the Jacobi equation, this yields:
\begin{lemma}[e.g., \cite{Es77}]
Given any non-singular Jacobi field $a(t)$ along $\g$, any other Jacobi field $z(t)$ along $\g$ can be written as 
$$z(t) = a(t) \left[ c_1 \int_0^t a^{-2}(s)ds + c_2 \right]$$
for some constants $c_1, c_2$.
\end{lemma}  
\nin The constants $c_1$ and $c_2$ can be determined by the initial conditions on $z$; for the Jacobi field $b$, we have $c_1=1$ and $c_2=0$.

Let $C>0$ be a constant such that, for $\e>0$ small enough, we can choose three positive $C^\infty$ functions $\psi_i: [0,1] \rightarrow \R$ with the following properties:\\
\begin{center}
{\def\arraystretch{1.4}
\begin{tabular}{|c|c|c|}
\hline
$\psi_1$                                                                            & $\psi_2$                                                                          & $\psi_3$          \\ \hline
$\supp(\psi_1) \subseteq \left] \frac{3}{4},1 \right]$   & $\supp(\psi_2) \subseteq \left] \frac{3}{4},1 \right]$ & $\supp(\psi_3) \subseteq \left] \frac{1}{4},\frac{1}{2} \right[$ \\
$a-\psi_1 > 0$                                                                 & $a-\psi_2 > 0$                                                                &  $a-\psi_3 > 0$          \\
$\psi_1'(1)=\e$                                                                 & $\psi_2(1)=\e, \; \psi_2'(1)=0$                                    & $\ds \int_0^1\psi_3 \, ds = \e$       \\
$\norm{\psi_1}_{C^2}\leq C\e$                                     & $\norm{\psi_2}_{C^2}\leq C\e$                                  & $\norm{\psi_3}_{C^2}\leq C\e $       \\
\includegraphics[width=1.35in]{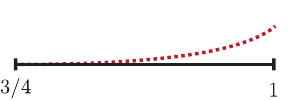}             & \includegraphics[width=1.35in]{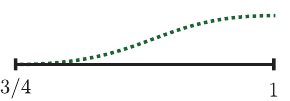}             & \includegraphics[width=1.35in]{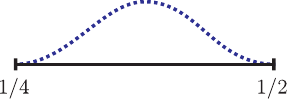}   \\  \hline
\end{tabular} }
\end{center} 
\vspace{.5cm}
Note that it is possible to find such functions because the support of $\psi_i$ is independent of $\e$.

Let $\ta_i = a-\psi_i$, as shown in Figure 1. Declaring $\ta_i$ to be a Jacobi field along $\g$ determines a new curvature function $\tk_i$ along $\g$ and a Jacobi field $\tb_i$ with initial conditions $\tb_i(0)=0$, $\tb_i'(0)=1$. The perturbation $\ta_1$ has the property that
$$\ta_1'(1)-a'(1) = \e.$$
The perturbation $\ta_2$ satisfies $\ta_2'(1) = a'(1)$ while 
$$\ta_2(1)-a(1) = \e.$$
The perturbation $\ta_3$ has the properties $\ta_3'(1) = a'(1)$ and $\ta_3(1) = a(1)$, while
\begin{align}
\tb_3(1)-b(1) &= \ta_3(1) \int_0^1 \ta_3^{-2}(s) \, ds - a(1) \int_0^1 a^{-2}(s) \, ds \notag\\
       &= a(1) \int_0^1 \frac{2a\psi_3-\psi_3^2}{a^2(a-\psi_3)^2} \, ds \notag\\
       &= a(1) \int_0^1 \frac{2a-\psi_3}{a^2(a-\psi_3)^2} \, \psi \; ds \notag\\
       &\geq a(1) \int_0^1 \frac{a}{a^4}\,\psi_3 \; ds = a(1) \int_0^1 \frac{1}{a^3}\,\psi_3 \; ds \notag\\
       &\geq \frac{a_{lb}}{a_{ub}^3} \,\e ,\notag
\end{align}
where $a_{ub}$ denotes an upper bound and $a_{lb}$ a lower bound of $a$ on $\g$, uniformly chosen over geodesic segments of length $1$.  Thus declaring $\ta_3$ a Jacobi field perturbs $b(1)$ by at least a constant times $\e$, with the constant depending only on the upper and lower bounds of $a$ along $\g([0,1])$.

\begin{figure}[ht]\label{perturbations}
\includegraphics[width=.7\textwidth]{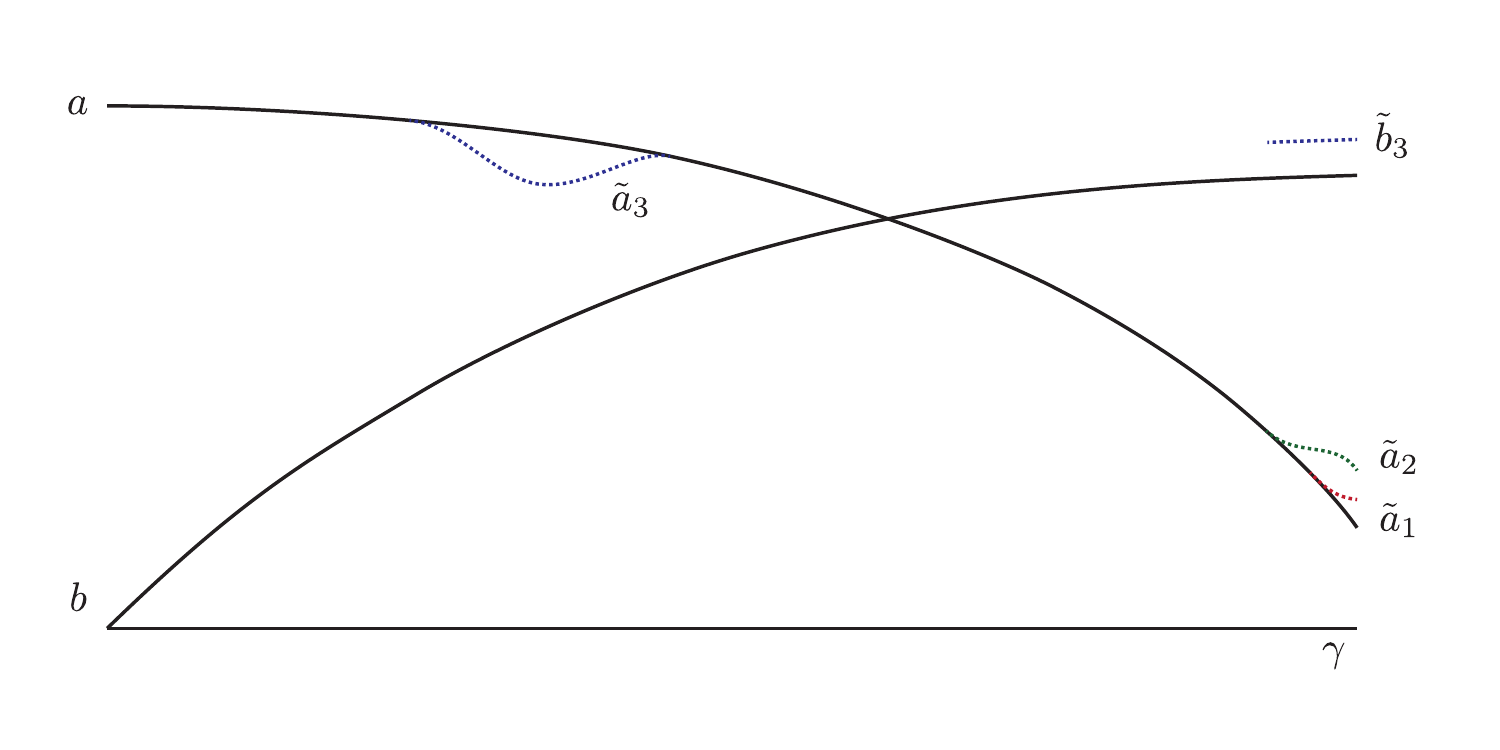}
\caption[Perturbing Jacobi fields $a$ and $b$]{Perturbing Jacobi fields $a$ and $b$.}
\end{figure}

We claim that the perturbation to the curvature function along $\g$ that produces the above perturbations to Jacobi fields is small. Consider the curvature $k(t)=K(t,0)$ along $\g$.  Declaring $\ta_i$ to be a Jacobi field gives a new curvature $\tk_i$ along $\g$.  From the Jacobi equation we have $$\tk_i = -\frac{\ta_i''}{\ta_i} = -\frac{a''-\psi_i''}{a-\psi_i}.$$  Then
\begin{align}
\norm{\tk_i-k}_{C^0} &= \Norm{ \frac{-(a''-\psi_i'')-(a-\psi_i)k}{a-\psi_i} }_{C^0} \notag\\
                                     &= \Norm{ \frac{-a''+\psi_i''+a''+\psi_i k}{a-\psi_i} }_{C^0} \notag\\
                                     &= \Norm{ \frac{\psi_i''+\psi_i k}{a-\psi_i} }_{C^0} \notag\\
                                     &\leq \frac{(1+\norm{k}_{C^0})C}{(a-\psi_i)_{lb}} \,\e, \notag
\end{align}
where $(a-\psi_i)_{lb}$ denotes a lower bound of $a-\psi_i$ on $\g$, uniform over such geodesic segments (e.g. for small enough $\e$, can take $(a-\psi_i)_{lb} = a_{lb} - C\e>0$).  Hence we need only perturb $k$ by $O(\e)$, with the constant depending only on the curvature $k$, the constant $C$, and the lower bound of $a-\psi_i$.

The Inverse Function Theorem now produces a ball of some radius $\de>0$ about $DP$.  Let $\sK(\g)$ be the Banach space of Gaussian curvatures along $\g$ (equivalently, the space of continuous functions on $[0,1]$), and let $\Phi: \sK(\g) \rightarrow  Sp(1)$ be the map assigning to each curvature $\tk$ the map $DP(\g,\tg)$, where $\tg$ is a metric for which $\g$ is a geodesic and $\tk$ is the curvature of $\tg$ along $\g$. (This is well--defined because $DP$ is determined by the Jacobi fields $a$ and $b$, whose values are determined by the curvature.) The three one-parameter families $\si_i(k; s) = k + s (\tk_i - k)$, for $i=1,2,3$, define a $3$-dimensional subspace $\fS \subset \sK(\g)$ since $\tk_i - k$ are independent functions over $\R$.\footnote{It is clear $\tk_3-k$ is independent from the other two, since the perturbation is supported on a different interval.  To see that $\tk_2-k$ and $\tk_1-k$ are independent over $\R$, note that $\tk_1-k$ is nonvanishing, while $\tk_2-k$ must vanish at some point in $]\frac{3}{4}, 1[$ in order to make $\psi_2'(1)=0$.}  Then the derivative of $\Phi|_\fS$ at $k$, in coordinates determined by the curves $\si_i$ on $\fS \subset \sK(\g)$ and $a'(1), a(1), b(1)$ on $Sp(1)$, is given by
$$(D\Phi|_\fS)k = 
\left[
\begin{array}{p{.5cm} p{.5cm} p{1.8cm}}
$\e$ & $0$ & $\;\;\;\;\;\;\; 0$ \\
$*$ & $\e$ & $\;\;\;\;\;\;\; 0$ \\
$*$ & $*$ & $\tb_3(1)-b(1)$ 
\end{array}
\right]. $$
By the Inverse Function Theorem, $\Phi|_\fS$ is a local diffeomorphism, so that the image of a neighborhood of $k$ under $\Phi|_\fS$ contains a ball of some radius $\de>0$ about $DP=DP(\g, g)$ in $Sp(1)$. Moreover, $\de$ can be set independently of $\g$ because all of the above computations depend continuously on the curvature $k$ and the space of geodesic segments on $S$ of length one is compact.

%%%%%
\subsection{Perturbing $k$ by perturbing $g$}
\label{createmetric}

For any curvature $\tk$ $C^0$--close to $k$, we can construct a metric $\tg$ supported in a tubular neighborhood $W$ of $\g$ with curvature $\tk$ along $\g$. This uses a well--known construction using Jacobi fields on the geodesics eminating perpendicularly from $\g$. We will show that $\tg$ is $C^2$--close to $g$, and that this distance is independent of $W$. First, extend $\tk$ to the Fermi coordinate neighborhood $U$ by setting $\hK(t,x) = \tk(t)$ (i.e., the new curvature is constant in the $x$-direction).  For each $t$, let $\hJ_t(x)$ be the Jacobi field satisfying $\hJ_t''(x)+\hK(t,x)\hJ_t(x)=0$ with initial conditions $\hJ_t(0)=1$, $\hJ_t'(0)=0$, and similarly for $J_t(x)$ with $K(t,x)$ (the curvature for the metric $g$).  In Fermi coordinates the metric $g$ takes the form 
$$g(t, x) = \left[\begin{array}{cc} J_t(x)^2 & 0 \\ 0 & 1\end{array}\right]$$
on $U$.  The desired metric around $\g$ is
$$\hg(t, x) = \left[\begin{array}{cc} \hJ_t(x)^2 & 0 \\ 0 & 1\end{array}\right],$$
which we will interpolate with $g$ to get a metric $\tg$.  For notational convenience, set $\Jd = \hJ_t(x) - J_t(x)$ and let $\Norm{\Jd}_{C^r, W}$ be the $C^r$ norm of $\Jd |_W$.  Let $\varphi$ be a $C^2$ bump function such that
$$\varphi(x)=\begin{cases} 1 & |x| < 1/4 \\ 0 & |x| > 1, \end{cases}$$ 
and set $\varphi_\eta(x) = \varphi (x/\eta)$.  Then for the tubular neighborhood  $W = [0,1] \times (-\eta,\eta) \subset U$, define a new metric $\tg$ with $\supp (\tg - g) \subset W$ by
$$\tg_{00}(t,x) = \left[(1-\varphi_\eta(x))J_t(x) + \varphi_\eta(x)\hJ_t(x)\right]^2$$ 
$$\tg_{10}(t,x) = \tg_{01}(t,x) = g_{01}(t,x)$$ 
$$\tg_{11}(t,x) = g_{11}(t,x).$$

Notice that a smaller tubular neighborhood (i.e. smaller $\eta$) means larger $C^2$ norm of $\varphi_\eta$, but smaller $\Norm{\Jd}_{C^r,W}$.  These effects cancel, as demonstrated below by calculating estimates of $\norm{\tg - g}_{C^2}$ and showing that this quantity does not depend on $W$.  One the one hand, note that (for $\eta<1$)
$$\norm{\varphi_\eta(x)}_{C^0} = \norm{\varphi(x)}_{C^0},$$
$$\norm{\varphi_\eta(x)}_{C^1} \leq \eta^{-1}\norm{\varphi(x)}_{C^1},$$
$$\norm{\varphi_\eta(x)}_{C^2} \leq \eta^{-2}\norm{\varphi(x)}_{C^2}.$$
On the other hand,

\begin{lemma}\label{comp}
For $\eta$ small enough,
$$\norm{\Jd}_{C^0,W} \leq 2 \eta^2 \norm{\tk-k}_{C^0},$$
$$\norm{\Jd}_{C^1,W} \leq 2 \eta \norm{\tk-k}_{C^0},$$
$$\norm{\Jd}_{C^2,W} \leq 2 \norm{\tk-k}_{C^0}.$$
\end{lemma}

\begin{proof}
First, we claim that $\Jd = \hJ_t(x) - J_t(x)$ is a $C^2$ function on $U$.  Since $g$ is a $C^4$ metric, $K$ is $C^2$ so that $J_t(x)$ is $C^2$.  Moreover, $a$ is $C^4$ and $\psi$ is $C^\infty$, so that $\tk=-\frac{a''-\psi''}{a-\psi}$ is $C^2$.  Thus $\hK$ is $C^2$ and $\hJ_t(x)$ is $C^2$, so $\Jd$ is $C^2$.  Then $D^2\Jd$ is continuous on $U$, and we can choose $W$ thin enough so that $$\norm{D^2\Jd}_{C^0,W} \leq 2 \norm{\tk-k}_{C^0}.$$  Hence $D\Jd$ grows at a rate at most $2 \norm{\tk-k}_{C^0}$ in the $x$-direction on $W$, and 
\begin{align}
\norm{D\Jd}_{C^0,W} &\leq 2 \eta \norm{\tk-k}_{C^0} + \norm{D\Jd}_{C^0,\g} \notag\\
                                       &= 2 \eta \norm{\tk-k}_{C^0}. \notag
\end{align}
Similar reasoning shows
\begin{align}
\norm{\Jd}_{C^0,W} &\leq 2 \eta^2 \norm{\tk-k}_{C^0} + \norm{\Jd}_{C^0,\g} \notag\\
                                     &= 2 \eta^2 \norm{\tk-k}_{C^0}. \notag
\end{align}
\end{proof}

\nin Recall that for $C^2$ functions $f$ and $g$, we have
$$\norm{f\cdot g}_{C^2} \leq \norm{f}_{C^2}\norm{g}_{C^0} + 2 \norm{f}_{C^1}\norm{g}_{C^1} + \norm{f}_{C^0}\norm{g}_{C^2}.$$
Then, by the estimates above,
\begin{align}
\norm{\tg - g}_{C^2} &= \norm{\tg_{00}-g_{00}}_{C^2} \notag\\
                           &= \norm{\varphi_\eta(x) \Jd}_{C^2} \notag\\
                           &\leq  \norm{\varphi_\eta(x)}_{C^0} \norm{\Jd}_{C^2,W} + 2 \norm{\varphi_\eta(x)}_{C^1} \norm{\Jd}_{C^1,W}  \notag\\
                           & \hspace{1cm} + \norm{\varphi_\eta(x)}_{C^2} \norm{\Jd}_{C^0,W} \notag\\  
                           &\leq 8 \norm{\varphi(x)}_{C^2} \norm{\tk-k}_{C^0} \notag
\end{align}
which does not depend on $\eta$.  Moreover, since $\norm{\tk-k}_{C^0} = O(\e)$, this means that $\norm{\tg-g}_{C^2} \leq O(\e)$ so that $\tg$ is a $C^2$-small perturbation of $g$.

%%%%%
\subsection{Avoiding a finite number of transverse geodesics}
\label{avoid}

Let $\sF = \{ \xi_1, \ldots, \xi_n \}$ be a finite set of geodesic segments that are transverse to $\g$, with $\xi_i$ intersecting $\g$ at the point $\xi_i^*$. Since there are finitely many segments, the angles at which the $\xi_i$ intersect $\g$ are bounded below. Then, for thin enough $W$, we can avoid perturbing the metric in a neighborhood $V$ of $F$ (using the construction above) by not perturbing the curvature $k$ in a neighborhood $V^*$ of the points $\{ \xi_1^*, \ldots, \xi_n^*\}$ on $\g$. Moreover, $V^*$ can be made arbitrarily small by shrinking $W$ (see Figure 2). Thus, retaining the ability to perturb $DP$ a uniform distance while making $\supp(\tg-g) \subset W \setminus V$ is a consequence of the following lemma (applied to $\tk$ and a curvature $\tk_1$ equal to $k$ on $V^*$ and $\tk$ outside a small neighborhood of $V^*$).

\begin{figure}[h]\label{intersections}
\centering
\includegraphics[width=4.5in]{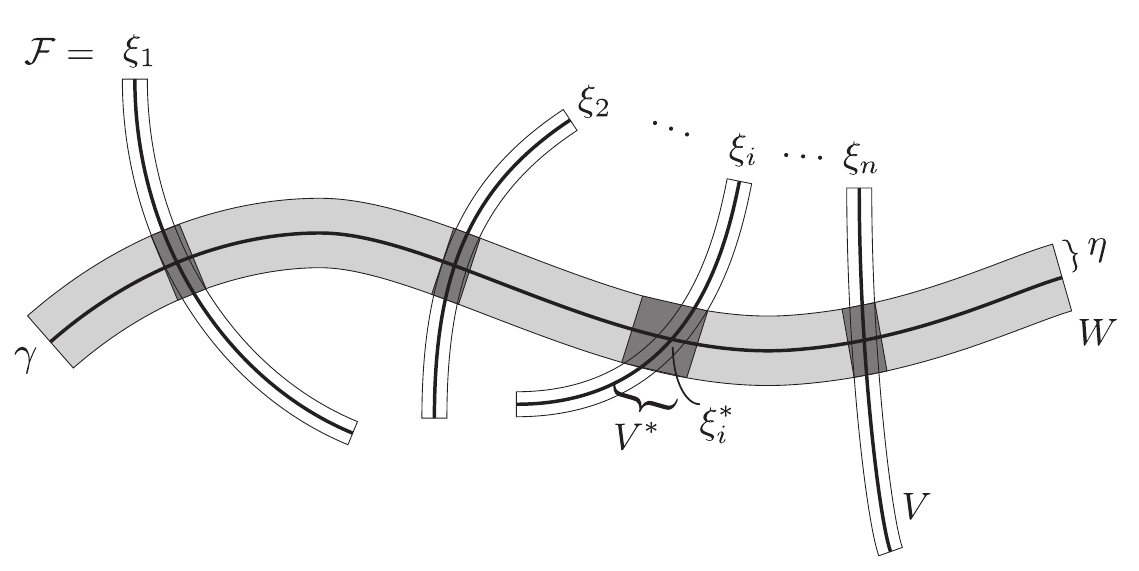}
\caption[]{Avoiding perturbing $g$ in neighborhoods of a finite number of geodesic segments.}
\end{figure}

\begin{lemma}
Let $\g$ be as above, and $C$ some constant. For any $\e>0$ and any curvature $k_1$ along $\g$ with $\norm{k_1-k}_{C^0} < C$, there exists a $\de>0$ (depending on $\e$) such that if $\supp(k_1-k) \subset \g$ is contained in a set of Lebesgue measure $\de$, then $\norm{DP_1-DP}<\e$.
\end{lemma}

\begin{proof}
Let $j$ be a Jacobi field for $k$ along $\g$, and $j_1$ a Jacobi field for $k_1$ with the same initial conditions. From the two Jacobi equations, we get
$$(j-j_1)'' + k (j-j_1) = (k_1-k)j_1,$$
which is a perterbation of the Jacobi equation for $k$ by $g(t) = (k_1-k)(t)j_1(t)$. That is, $y(t) = j(t)-j_1(t)$ is a solution to the non-homogeneous second order equation
$$y'' + k(t) y = g(t),$$ 
with $y(0) = y'(0) = 0$. Since the Jacobi fields $a$ and $b$ are independent solutions\footnote{This means that the Wronskian $W(a,b)= ab'-a'b \not= 0$; in this case $W(a,b) \equiv 1$.} to the corresponding homogeneous second order equation $y''+ky=0$, using a variation of parameters yields the solution
$$y(t) = - a(t) \int_0^t b(s) g(s) \, ds + b(t) \int_0^t a(s) g(s) \, ds.$$
Since all of these functions are bounded, both $y(t)$ and $y'(t)$ can be made arbitrarily small by making the support of $k_1-k$ (and thus the support of $g(t)$) arbitrarily small.
\end{proof}

%%%%%%%%%%%%%%%%%%%%%
%  FRANKS' LEMMA FOR HIGHER DIMENSIONAL GEODESIC FLOWS
%%%%%%%%%%%%%%%%%%%%%
\medskip
\section{Franks' lemma for higher--dimensional geodesic flows}
\label{FLGF-HDS}
%%%%%%%

This section generalizes the techniques of the previous section to give a proof of Franks' lemma for geodesic flows in higher dimensions, as found in~\cite{Con10}. The relevant spaces, along with their dimension, are given below:
\begin{center}
{\def\arraystretch{1.2}
\begin{tabular}{| r c | c | c |}
\hline
&                         & \multicolumn{2}{ c|}{Dimension} \\ 
\multicolumn{2}{| c |}{Space}                       & \multicolumn{1}{c}{Surface} & $M^{n+1}$ \\ \hline
                                             &$M$                 & $2$        & $n+1$ \\ 
sphere bundle over $M$ &$SM$              & $3$        &  $2n+1$ \\
hypersurface in $SM$      &$\Si$               & $2$        &  $2n$ \\
%&$\Mat_{2n} \R$   & $4$        &  $4n^2$ \\
symplectic $2n \times 2n$ matrices &$Sp(n)$          & $3$        &  $2n^2+n$ \\ 
symmetric $n \times n$ matrices &$\sS(n)$       & $1$        &  $\frac{1}{2}(n^2+n)$ \\
orthogonal $n \times n$ matrices &$O(n)$            & $0$        &  $\frac{1}{2}(n^2-n)$  \\ \hline
\end{tabular} }
\end{center}

\begin{theorem2}[\cite{Con10}]
Let $g \in \sG^4(M) \cap \sG_1$ and let $\sU$ be a neighborhood of $g$ in $\sG^2(M)$.  Then there exists $\de = \de(g,\sU) >0$ such that for any simple geodesic segment $\g$ of length $1$, each element of $B(DP(\g,g),\de) \subset Sp(n)$ is realizable as $DP(\g, \tg)$ for some $\tg \in \sU \cap \sG_\g(M)$.  Moreover, for any tubular neighborhood $W$ of $\g$ and any finite set $\sF$ of transverse geodesics, the support of the perturbation can be contained in $W \setminus V$ for some small neighborhood $V$ of the transverse geodesics $\sF$.
\end{theorem2}

Let $\g$ be a non--intersecting geodesic segment of length $1$ and fix a set of Fermi coordinates $(t; x_1, \ldots, x_n )$ along $\g$ with coordinate neighborhood $U$.  The map $DP=DP(\g,g;0,1)$ takes the following form on Jacobi fields along $\g$:
$$DP_{2n\times 2n} \left[ \begin{array}{c} J(0) \\ J'(0)  \end{array} \right]_{2n\times 1}  =  \left[ \begin{array}{c} J(1) \\ J'(1)  \end{array} \right]_{2n\times 1} .$$
The $2n$-dimensional space of Jacobi fields has a basis given by the column vectors of the $n \times n$ matrices $A(t)$ and $B(t)$ with initial conditions
$$A(0) = \Id_{n\times n}, \;\;\; A'(0) = 0_{n\times n}$$
and
$$B(0) = 0_{n\times n}, \;\;\; B'(0) = \Id_{n\times n}.$$
This allows us to write down $DP$ in coordinates:
$$DP = \left[ \begin{array}{c c} A(1) & B(1) \\ A'(1) & B'(1)  \end{array} \right]_{2n\times 2n}.$$
Below, we consider linear Poincar\'e maps from $\Si_{\dg(a)}$ to $\Si_{\dg(b)}$ for varying $a,b \in [0,1]$.  When writing down the matrix $DP(\g, g;a,b)$, we will make a time shift so that $a=0$; then $A$ and $B$ give bases for Lagrangian subspaces of Jacobi fields defined with initial conditions at $0$ as above, and 
$$DP(t) = DP(\g, g; a,a+t) = \left[ \begin{array}{c c} A(t) & B(t) \\ A'(t) & B'(t)  \end{array} \right]_{2n\times 2n}.$$
This has the notational advantage that $DP(t_1+t_2)=DP(t_2)DP(t_1)$, where $DP(t_2)$ is understood to be $DP(\g,g;t_1,t_1+t_2)$. 

We wish to find $\dim Sp(n) = 2n^2 + n$ curves in $\sG^2(M)$ such that their images under the map $DP(\g, \cdot): \sG^2(M) \rightarrow Sp(n)$ span the tangent space at $DP$, and then use the Inverse Function Theorem to produce an open ball in $Sp(n)$. In higher dimensions, the fact that the curvature matrix $R=g(R(\cdot, \dg)\dg, \cdot)$ is symmetric and must remain so under perturbation imposes a non--trivial restriction on how $A$ can be perturbed, via the equation $\tR = \tA'' \tA^{-1}$. 

In fact, it is not obvious how to perturb $A$ while keeping $\tR = \tA'' \tA^{-1}$ symmetric, so we consider instead $U_A = A' A^{-1}$. Differentiating and employing the Jacobi equation shows that $U_A(t)$ satisfies the Riccati equation $$U_A'+U_A^2+R = 0.$$ Working with the Riccati equation has the advantage that making a symmetric perturbation to $U_A$ guarantees (in fact, is equivalent to) that the perturbed curvature will be symmetric.

As described below, the three families of perturbations from Section \ref{FLGF-SS} generalize to give $3 \cdot \dim \sS(n) = \frac{3}{2}(n^2+n)$ one-parameter families of perturbations to $DP$. This leaves $\frac{1}{2}(n^2-n) = \dim O(n)$ dimensions to fill in $Sp(n)$, which we do by making two perturbations at different points along $\g$ that cancel each other out modulo the effects of the dynamics along $\g$ in between these points.  This is possible as long as there are points along $\g$ for which the matrix $R$ has distinct eigenvalues. Geometrically, this can be seen as supplying some rotation in the dynamics along $\g$.

%%%%%%%%%%%%%%%%%%%%%
%  The set G_1
%%%%%%%%%%%%%%%%%%%%%
\subsection{The set $\sG_1$}
Let $\sG_1$ be the set of metrics for which every geodesic segment of length $\frac{1}{2}$ has some point at which the curvature matrix $R$ has distinct eigenvalues.  More precisely, let $h: \sS(n) \rightarrow \R_{\geq 0}$ be given by 
$$h(R) = \prod_{1 \leq i < j \leq n} (\la_j - \la_i),$$
where $\la_1 \leq \la_2 \leq \ldots \leq \la_n$ are the eigenvalues of $R$.  It is evident that $h(R) \geq 0$, and $h(R)=0$ if and only if $R$ has repeated eigenvalues.  Let $H: \sG^2(M) \rightarrow \R_{\geq 0}$ be the smallest value of $h$ over all length--$\frac{1}{2}$ geodesics on $SM$:
$$H(g) = \min_{\theta \in SM} \max_{t \in [0,\frac{1}{2}]} h(R(\phi_g^t(\theta))).$$  
Denote the set of metrics for which this number is strictly positive by $\sG_1 = \{ g \in \sG^2(M) | H(g) > 0 \}.$ Theorem 6.1 of \cite{Con10} states that $H:  \sG^2(M) \rightarrow \R_{\geq 0}$ is continuous and that $\sG_1$ is $C^2$ open and $C^\infty$ dense in $\sG^2(M)$. This means that for a $C^2$ open and dense set of metrics, any geodesic segment of length $\frac{1}{2}$ has a point along it where the eigenvalues of $R$ have at least a certain amount of separation, depending only on the metric $g$. We need this property when assembling perturbations in Section \ref{schema}.

%%%%%
\subsection{Perturbing $DP$ by perturbing the curvature matrix $R$}

%%%%%%%%%%%%%%%%%%%%%
%  Estimates
%%%%%%%%%%%%%%%%%%%%%
First, we need to make explicit how perturbations to $U_A = \psi$ and $DP$ are related. $A(t)$, $A'(t)$ and $U_A(t)$ satisfy the equations
$$A'(t) = U_A(t) A(t)$$
and 
$$A(t) = A(0) + \int_0^t U_A(s) A(s) \, ds.$$ 
For $\tU_A = U_A + \psi$, let $\tA(t)$ be the resulting perturbation of $A(t)$ and write $\De A(t) = \tA(t) - A(t)$; similarly for $\tB(t), \tDP(t), \De B(t),$ and $\De DP(t)$.  Then 
\begin{align}
\De A'(t) = \tA'(t) - A'(t) &= \tU_A \tA - U_A A \notag \\
                            &= (U_A + \psi)(A + \De A) - U_A A \notag\\
                            &= \psi A + (U_A + \psi) \De A, \label{A'}
\end{align}
and
\begin{equation}
\De A(t) = \int_0^t \De A'(s) \, ds = \int_0^t \psi A (s) \, ds + \int_0^t (U_A + \psi) \De A (s) \, ds. \label{A}
\end{equation}

From the data $A(t)$, we can also write down $B(t)$.  As in Section \ref{FLGF-SS}, reduction of order on the Jacobi equation $A'' = -R A$ gives:
$$B(t) = A(t) \int_0^t (A^T A)^{-1}(s) \, ds.$$
(See, for instance, \cite{Es77}.) Then
\begin{align}
\De B(t) &= \tB(t) - B(t) \notag \\
               &= \De A(t) \int_0^t (\tA^T \tA)^{-1} (s) \, ds + A(t) \int_0^t \left( (\tA^T \tA)^{-1}(s) - (A^T A)^{-1}(s) \right) \, ds. \label{B}
\end{align}
Notice that $A, A',$ and $B$ determine $B'$, since differentiating the above formula for $B$ yields
\begin{align*}
B'(t) &= A'(t) A^{-1}(t) B(t) + (A^T)^{-1}(t) \\
        &= U_A(t) B(t) + (A^T)^{-1}(t). 
\end{align*}
Then
\begin{equation}
\De B'(t) = \tB'(t) - B'(t) 
               = U_A \De B + \psi(B \De B) + (\tA^T)^{-1} - (A^T)^{-1}.  \label{B'}
\end{equation}
This describes the relation between perturbing $U_A$ and $DP$.

$R$ and $\sU_A$ are related via the Riccati equation $R+U_A^2+U_A' = 0$.  Declaring $\tU_A$ to satisfy this equation yields a new curvature matrix $\tR$, and
\begin{equation}\label{R}
\De R(t) =\tR(t) - R(t) = - \psi' - U_A \psi - \psi U_A - \psi^2.
\end{equation}

Let $\| \cdot \| : \Mat(n) \rightarrow \R$ be the matrix norm defined by $\|A \| = \max_j \sum_i |a_{ij}|$, which is the maximum of the column vector sums. This norm is submultiplicative, which is used extensively in the estimates below. For a matrix $M$ depending on $\e$, write $M = O(\e)$ if $\Norm{M} \leq C \e$ for a constant $C$, and $M = \Theta(\e)$ if $c\e \leq \Norm{M} \leq C \e$ for some constants $c$ and $C$.  We will use $\Theta(\e)$ rather than $O(\e)$ to indicate that some entry of the matrix $M$ has size bounded from below by $c\e$; in particular, $\Norm{M}$ is not too small.  If we are only concerned with a general size estimate of $\De DP(t)$ resulting from a change to the curvature of size $\De R=O(\e)$, then the Jacobi equation along with the initial conditions for $A(t)$ and $B(t)$ yield
\begin{equation}\label{gen}
\begin{array}{l l}
\De A''(t) = O(\e) & \De B''(t) = O(\e t) \\
\De A'(t) = O(\e t) & \De B'(t) = O(\e t^2) \\
\De A(t) = O(\e t^2) & \De B(t) = O(\e t^3). 
\end{array}
\end{equation}

%%%%%%%%%%%%%%%%%%%%%
%  Perturbation functions
%%%%%%%%%%%%%%%%%%%%%
\subsubsection{Perturbation functions and their effects}
Consider $C^{\infty}$ functions satisfying the following properties:
\begin{center}
{\def\arraystretch{1.4}
\begin{tabular}{|c|c|c|}
\hline
$\psi_1: ]-\infty,\de^3] \rightarrow \R$                           & $\psi_2: \R \rightarrow \R$                                       & $\psi_3: \R \rightarrow \R$                          \\ 
\hline
$\supp(\psi_1) \subseteq \left] 0,\de^3 \right]$           & $\supp(\psi_2) \subseteq \left] 0, \de^{3/2} \right[$ & $\supp(\psi_3) \subseteq \left] 0, \de \right[$ \\
$\norm{\psi_1}_{C^0}\leq \e \de^3$                             & $\norm{\psi_2}_{C^0} \leq C \e \de^{3/2}$              & $\norm{\psi_3}_{C^0} \leq C \e \de$               \\ 
$\norm{\psi_1}_{C^1}\leq C\e$                                     & $\norm{\psi_2}_{C^1}\leq C\e$                                 & $\norm{\psi_3}_{C^1}\leq C\e$        \\
%$\det(\sU_A-\psi_1) > 0$                                                & $\det(\sU_A-\psi_2) > 0$                                              &                                      \\
$\psi_1(\de^3)=\e \de^3$                                               & $\psi_2(\de^{3/2})=0$                                                 & $\psi_3(\de)=0$      \\
                                                                                            & $\int_0^{\de^{3/2}} \psi_2(t)\, dt=\e \de^3$              & $\int_0^\de \psi_3(t)\, dt=0$  \\
                                                                                            &                                                                                         & $\int_0^\de \int_0^t \psi_3(s) \, ds \, dt = \e \de^3$ \\ 
\includegraphics[width=1.35in]{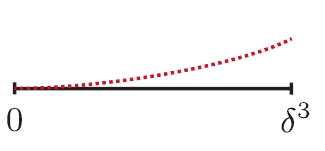}                      & \includegraphics[width=1.35in]{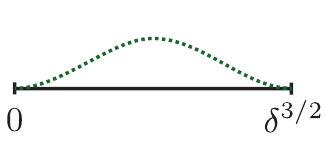}                  & \includegraphics[width=1.35in]{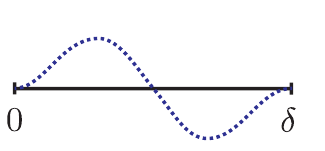}   \\  \hline
\end{tabular} }
\end{center}
and let $\psi^{ij}_k$ be the symmetric $n \times n$ matrix $M$ with $m_{ij} = m_{ji} = \psi_k$ and $0$ otherwise.  Write $\Psi_k(t) = \int_0^t \psi_k(s) \, ds$.

Heuristically, the following computational lemmas show that, over the above intervals of support, adding $\psi_1$ to $U_A$ perturbs $A'$, adding $\psi_2$ to $U_A$ perturbs $A$, and adding $\psi_3$ to $U_A$ perturbs $B$. The estimates follow from relatively straight--forward applications of Equations \ref{A'}--\ref{R}. Note that the differing sizes of support of the perturbations are in order that their effects are of the same size.

%%%%%%%%%%%%%%%%%%%%%
%  perturbation I estimates
%%%%%%%%%%%%%%%%%%%%%
\begin{lemma}
Let $\tU_A = U_A + \psi^{ij}_1$. Then
$$\De DP(\de^3) = \left[ \begin{array}{c c} \De A(\de^3) & \De B(\de^3) \\ \De A'(\de^3) & \De B'(\de^3) \end{array} \right] = \left[ \begin{array}{c c} O(\e\de^6) & O(\e\de^9) \\ \Theta(\e\de^3) & O(\e\de^6) \end{array} \right] , $$
with $$\De A'(\de^3) = \psi^{ij}_1(\de^3) + O(\e \de^9),$$
and $\De R =O(\e).$
\end{lemma}

\begin{proof}
From the Riccati equation and the initial conditions for $U_A$, we get $U_A(t) = O(t)$ on $[0,1]$.  Then by Equation \ref{R},
\begin{align*}
\De R(t) &= \psi'(t) + U(t) \psi (t) + \psi (t) U (t) + \psi^2 (t) \\
               &= O(\e) + O(\de^3) \cdot O(\e\de^3) + O(\e\de^3) \cdot O(\de^3) + O(\e^2\de^6) = O(\e).
\end{align*}
By Equation \ref{A'},
\begin{align*}
\De A'(\de^3)  &= \psi(\de^3) A(\de^3) + \left( U_A(\de^3) + \psi(\de^3) \right) \De A(\de^3)) \\
                         &= \psi(\de^3) A(\de^3) + \left( O(\de^3)+O(\e\de^3) \right) O(\e\de^6) \\
                         &= \psi(\de^3) A(\de^3) + O(\e \de^9).
\end{align*}
Write $A(\de^3) = A(0) + (A(\de^3) - A(0))$.  Since $\Norm{A'}_{[0,\de^3]} < C_A \de^3$, we have $(A(\de^3) - A(0)) = O(\de^6)$.  Then
\begin{align*}
\De A'(\de^3)  &= \psi(\de^3) A(\de^3) + O(\e \de^9) \\
                         &= \psi(\de^3) A(0) + \psi(\de^3) O(\de^6) + O(\e \de^9) \\
                         &= \psi(\de^3) + O(\e \de^9).
\end{align*}
Equation \ref{gen}, along with $\De R(t) = O(\e)$, give
\begin{align*}
\De A(\de^3) &= O(\e\de^6) \\
\De B(\de^3) &= O(\e\de^9) \\
\De B'(\de^3) &= O(\e\de^6). 
\end{align*}
\end{proof}

%%%%%%%%%%%%%%%%%%%%%
%  perturbation II estimates
%%%%%%%%%%%%%%%%%%%%%
\begin{lemma}
Let $\tU_A = U_A + \psi^{ij}_2$. Then
$$\De DP(\de^{3/2}) = \left[ \begin{array}{c c} \De A(\de^{3/2}) & \De B(\de^{3/2}) \\ \De A'(\de^{3/2}) & \De B'(\de^{3/2}) \end{array} \right] = \left[ \begin{array}{c c} \Theta(\e\de^3) & O(\e\de^{9/2}) \\ O(\e\de^{9/2}) & O(\e\de^3) \end{array} \right] , $$
with $$\De A(\de^{3/2}) = \Psi^{ij}_2(\de^{3/2}) + O(\e \de^6),$$
and $\De R = O(\e).$
\end{lemma}

\begin{proof}
Since $\Norm{\psi_2^{ij}}_{C^1} \leq C \e$, $\De R = O(\e)$.
From Equation \ref{A}, we have
\begin{align*}
\De A(\de^{3/2}) &= \int_0^{\de^{3/2}} \psi(s) A(s) \, ds + \int_0^{\de^{3/2}} (U_A + \psi) \De A (s) \, ds \\
                              &= \int_0^{\de^{3/2}} \psi(s) A(s) \, ds + \de^{3/2} \left( O(\de^{3/2}) + O(\e\de^{3/2}) \right) O(\e\de^3) \\
                              &= \int_0^{\de^{3/2}} \psi(s) A(0) \, ds + \int_0^{\de^{3/2}} \psi(s) (A(s)-A(0)) \, ds + O(\e\de^6) \\
                              &= \int_0^{\de^{3/2}} \psi(s) \, ds + O(\e\de^6),
\end{align*}
since $A(s)-A(0) = O(\de^3)$ on $[0,\de^{3/2}]$.
As $\psi_2^{ij}(\de^{3/2}) = 0$, Equation \ref{A'} gives
\begin{align*}
\De A'(\de^{3/2}) &= \psi(\de^{3/2}) A(\de^{3/2}) + \left( U_A(\de^{3/2}) + \psi(\de^{3/2}) \right) \De A(\de^{3/2}) \\
                              &= U_A(\de^{3/2}) \De A(\de^{3/2}) \\
                              &= O(\de^{3/2}) O(\e \de^3) = O(\e \de^{9/2}).
\end{align*}
Equation \ref{gen}, along with $\De R(t) = O(\e)$, give
\begin{align*}
\De B(\de^{3/2}) &= O(\e\de^{9/2}) \\
\De B'(\de^{3/2}) &= O(\e\de^3). 
\end{align*}
\end{proof}

%%%%%%%%%%%%%%%%%%%%%
%  perturbation III estimates
%%%%%%%%%%%%%%%%%%%%%
\begin{lemma} \label{III}
Let $\tU_A = U_A + \psi^{ij}_3$. Then
$$\De DP(\de) = \left[ \begin{array}{c c} \De A(\de) & \De B(\de) \\ \De A'(\de) & \De B'(\de) \end{array} \right] = \left[ \begin{array}{c c} O(\e\de^4) & \Theta(\e\de^3) \\ O(\e\de^5) & O(\e\de^4) \end{array} \right] , $$
with $$\De B(\de) = -2 \int_0^\de \Psi^{ij}_3(t) \, dt + O(\e\de^5),$$
and $\De R = O(\e).$
\end{lemma}

\begin{proof}
Since $\Psi_3^{ij}(\de) = 0$, Equation \ref{A} gives
\begin{align*}
\De A(\de) &= \int_0^{\de} \psi(s) A(s) \, ds + \int_0^{\de} (U_A + \psi) \De A (s) \, ds \\
                              &= \int_0^{\de} \psi(s) \, ds + \int_0^{\de} \psi(s) (A(s)-A(0)) \, ds + O(\e\de^4) \\
                              &= O(\e\de^4).
\end{align*}
Since $\psi_3^{ij}(\de) = 0$, Equation \ref{A'} and the above computation give
\begin{align*}
\De A'(\de) &=\psi(\de) A(\de) + \left( U_A(\de) + \psi(\de) \right) \De A(\de) \\
                    &= U_A(\de) \De A(\de) \\
                    &= O(\de) O(\e\de^4) = O(\e \de^5).
\end{align*}
From Equation \ref{B} and Lemma \ref{Dg} (below) we get
\begin{align*}
\De B(\de) &= \De A(\de) \int_0^\de (\tA^T \tA)^{-1} (s) \, ds + A(\de) \int_0^\de \left( (\tA^T \tA)^{-1}(s) - (A^T A)^{-1}(s) \right) \, ds \\
                   &= O(\e\de^5) + (\Id + O(\e\de^4)) \int_0^\de (-\De A(s) - \De A^T(s) + O(\e\de^4) ) \, ds \\
                   &= \int_0^\de -(\De A(s)+ \De A^T(s)) \, ds + O(\e\de^5) \\
                   &= -2 \int_0^\de \int_0^t \psi(s) \, ds \, dt + O(\e\de^5),
\end{align*}
while from Equation \ref{B'} and Lemma \ref{Dg} we get
\begin{align*}
\De B'(\de) &= U_A(\de) \De B(\de) + \psi^{ij}_3(\de)(B(\de) + \De B(\de)) + (\tA^T)^{-1}(\de) - (A^T)^{-1}(\de) \\
                    &= O(\de) O(\e\de^3) + 0 - \De A^T(\de) + O(\e\de^4) \\
                    &= O(\e\de^4).
\end{align*}
\end{proof}

The following technical lemma is necessary when giving estimates based on Equation \ref{B}, and is used above in computations of the proof of Lemma~\ref{III}. It says, roughly, that $\De ((A^T A)^{-1}) \approx -\De A - \De A^T$ and $\De((A^T)^{-1}) \approx -\De A^T$.

\begin{lemma}\label{Dg}
For $0 \leq s \leq \de$, 
$$(\tA^T \tA)^{-1}(s) - (A^T A)^{-1}(s) = - ( \De A(s) + \De A^T(s) ) + O(\e\de^4)$$
and
$$(\tA^T)^{-1}(s) - (A^T)^{-1}(s) = -\De A^T(s) + O(\e\de^4).$$
\end{lemma}

\begin{proof}
Let $g: GL(n) \rightarrow \sS(n)$ be given by $g(A) = (A^TA)^{-1}$.  We wish to compute $(\tA^T \tA)^{-1}(s) - (A^T A)^{-1}(s) = g(\tA(s)) - g(A(s))$, which we will do by integrating the derivative of $g$ along a path from $A(s)$ to $\tA(s)$.  Hence $$(\tA^T \tA)^{-1}(s) - (A^T A)^{-1}(s) = \int_0^1 D_X g (\De A(s)) \, dr,$$ where $X=(1-r)A(s) + r \tA(s) = A(s) + r \De A(s)$.  Let us compute $D_X g(Y)$ (we will apply this to $Y=\De A(s)$).  Write $g=i \circ h$, where $i(A)=A^{-1}$ and $h(A) = A^T A$.  Then 
\begin{align*}
D_X g(Y) &= D_{h(X)} i \circ D_X h(Y) \\
                  &= D_{X^T X} i (X^T Y + Y^T X) \\
                  &= -(X^T X)^{-1} (X^T Y + Y^T X) (X^T X)^{-1} \\
                  &= -(X^{-1} Y (X^T X)^{-1} + (X^T X)^{-1} Y^T (X^T)^{-1} ) \\
                  &= - \left[ (X^{-1} Y (X^T X)^{-1}) + (X^{-1} Y (X^T X)^{-1})^T \right],
\end{align*}
which is the symmetrization of $(X^{-1} Y (X^T X)^{-1})$.  Now, $X = A(s) + r \De A(s) = \Id + (A(s) - A(0)) + r \De A(s) = \Id + O(\de^2) + O(\e\de^2) = \Id + O(\de^2)$.  Then
\begin{align*}
X^{-1} Y (X^T X)^{-1} &= (\Id + O(\de^2))^{-1} Y (\Id + O(\de^2))^{-1} \\
                                      &= (\Id + O(\de^2)) Y (\Id + O(\de^2)) \\
                                      &= Y + O(\e\de^4),
\end{align*}
so that $\int_0^1 D_X g (\De A(s)) \, dr = - ( \De A(s) + \De A^T(s) ) + O(\e\de^4)$.

Similarly, for $f(A) = (A^T)^{-1}$, we have $(\tA^T)^{-1}(s) - (A^T)^{-1}(s) = \int_0^1 D_X f(\De A(s)) \, dr$ and
\begin{align*}
D_X f(Y) &= D_{T(X)} i \circ D_X T(Y) \\
                  &= D_{X^T}i (Y^T) \\
                  &= -(X^T)^{-1} (Y^T) (X^T)^{-1} = -(X^{-1} Y X^{-1})^T.
\end{align*}
Since $X = \Id + O(\de^2)$, this is $D_X f(Y) = -\De A^T(s) + O(\e\de^4)$.
\end{proof}

%%%%%%%%%%%%%%%%%%%%%
%  Placement of perturbations
%%%%%%%%%%%%%%%%%%%%%
\subsubsection{Perturbation schema} \label{schema}
Let $g \in \sG_1$ and consider a length $1$ piece of geodesic $\g$.  Let $t_0$ be the time for which $R(t_0)$ has distinct eigenvalues, with separation $| \la_i - \la_j |\geq H(g)$; we consider the map $DP$ over an interval $[t_0, t_0+d]$ with $\de \ll d \ll 1$, e.g. $\de = d^2$.  For the following, we will make a time shift so that $t_0 = 0$, and work with the following particular set of Fermi coordinates.   Since $R$ is symmetric, $R(0)$ can be diagonalized by an orthogonal matrix $Q$ to $$Q^{-1} R(0) Q = diag(\la_1, \ldots, \la_n),$$ with $\la_i$ distinct by assumption.  Let $v_1, \ldots v_n$ be the eigenvectors for $R(0)$; we will write $DP(d)$ in Fermi coordinates based on this set of orthonormal vectors in $T_{\dg(0)}\Si_{0}$.  Note that the map $DP(t)$ for the original metric, for small enough $t$, is $DP(t) = \Id + O(t)$.  

For Perturbation IV, we need a finer description of $DP(d)$.  The matrix $R$ is not constant along $\g$, but since $g$ is a $C^3$ metric there is a constant $C$ such that $\Norm{R'} \leq C$ on $M$.  Then $$R(t) = R(0) + P(t),$$ where $\Norm{P(t)}_{C^1} \leq C$.  In particular, since $P(0)=0$, we have $\Norm{P(t)}_{C^0} \leq C t$ on $[0,d]$.  Then, using the definitions of $A$ and $B$, we have
$$A(t) = diag( 1-\frac{\la_1}{2} t^2, \ldots, 1-\frac{\la_n}{2} t^2 ) + O(t^3)$$
$$A'(t) = diag( \la_1 t, \ldots, \la_n t ) + O(t^2)$$
$$B(t) = diag( t-\frac{\la_1}{6} t^3, \ldots, t-\frac{\la_n}{6} t^3 ) + O(t^4)$$
$$B'(t) = diag( 1-\frac{\la_1}{2} t^2, \ldots, 1-\frac{\la_n}{2} t^2 ) + O(t^3).$$
Hence
$$DP(t) = \left[ \begin{array}{c c} \Id & 0 \\ 0 & \Id \end{array} \right] + \left[ \begin{array}{c c} 0 & \Id \\ I(\la) & 0 \end{array} \right] t + O(t^2),$$
where $I(\la) = diag(\la_1, \ldots, \la_n)$.

We will perform the following families of perturbations to $\sU_A$ along $\g$ (see Figure 3). Note that the factor of $d$ in Perturbations I, II, and III is to make the size of the perterbation effect the same as that for Perturbation IV.

\begin{figure}[h]
\centering
\includegraphics[width=4in]{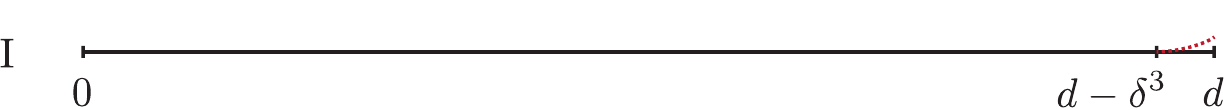} \\
\includegraphics[width=4in]{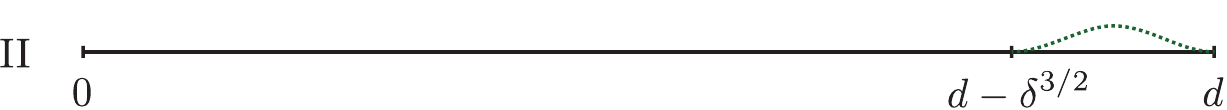} \\
\includegraphics[width=4in]{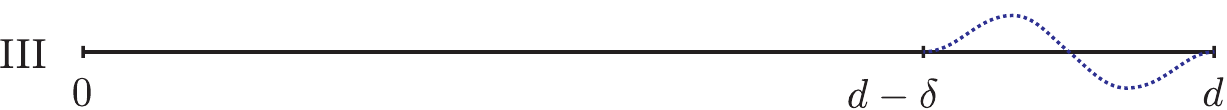} \\
\includegraphics[width=4in]{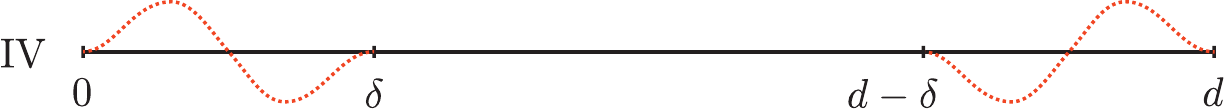}
\caption[Placement of perturbations I--IV]{Placement of perturbations I--IV.}
\end{figure}

%%%  Pert I
\nin {\bf Perturbation I.}  For $1 \leq i \leq j \leq n$, let $$\tU_I^{ij}(t) = U_A(t) + d \cdot \psi^{ij}_1(t-(d-\de^3)).$$
Then, using the fact that $DP(t_1+t_2) = DP(t_2) DP(t_1)$,
\begin{align*}
\De_I^{ij} DP(d) &= \De DP(\de^3) DP(d-\de^3) \\
                       &= d \left[ \begin{array}{c c} O(\e\de^6) & O(\e\de^9) \\ \psi^{ij}_1(\de^3) + O(\e \de^9) & O(\e\de^6) \end{array} \right]  (\Id + O(d)) \\
                       &= \left[ \begin{array}{c c} 0 & 0 \\ d \cdot \psi^{ij}_1(\de^3) & 0 \end{array} \right] + O(\e \de^3 d^2).
\end{align*}

%%% Pert II
\nin {\bf Perturbation II.}  For $1 \leq i \leq j \leq n$, let $$\tU_{II}^{ij}(t) = U_A(t) + d \cdot \psi^{ij}_2(t-(d-\de^{3/2})).$$
Then
\begin{align*}
\De_{II}^{ij} DP(d) &= \De DP(\de^{3/2}) DP(d-\de^{3/2}) \\
                            &= d \left[ \begin{array}{c c} \Psi^{ij}_2(\de^{3/2}) + O(\e \de^6) & O(\e\de^{9/2}) \\ O(\e\de^{9/2}) & O(\e\de^3) \end{array} \right] (\Id + O(d)) \\
                            &= \left[ \begin{array}{c c} d \cdot \Psi^{ij}_2(\de^{3/2}) & 0 \\ 0 & * \end{array} \right] + O(\e\de^3 d^2),
\end{align*}
where the $*$ is an entry of $O(\e\de^3 d)$ (this block will not be used when we put coordinates on $Sp(n)$).

%%%  Pert III
\nin {\bf Perturbation III.}  For $1 \leq i \leq j \leq n$, let $$\tU_{III}^{ij}(t) = U_A(t) + d \cdot \psi^{ij}_3(t-(d-\de)).$$
Then
\begin{align*}
\De_{III}^{ij} DP(d) &= \De DP(\de) DP(d-\de)\\
                           &= d \left[ \begin{array}{c c} O(\e\de^4) & \int_0^\de \Psi^{ij}_3(s) \, ds + O(\e\de^5) \\ O(\e\de^5) & O(\e\de^4) \end{array} \right] (\Id + O(d)) \\
                           &= \left[ \begin{array}{c c} 0 & d \cdot \int_0^\de \Psi^{ij}_3(s) \, ds \\ 0 & 0 \end{array} \right] + O(\e\de^3 d^2).
\end{align*}

%%%  Pert IV
The next perturbation makes use of the natural rotation of the dynamics when the curvature matrix $R$ has $n$ distinct eigenvalues. When $R \equiv 0$, for instance, the two ends of Perturbation IV cancel out and we are left with the original linear Poincar\'e map $DP$; however, because of the distinct eigenvalues of $R$, the effects of the initial perturbation rotate slightly before the end perturbation takes place, as in the Figure 4. This produces a perturbation with antisymmetric components to $A(d)$.
\begin{figure}[h]
\centering
\includegraphics[width=3in]{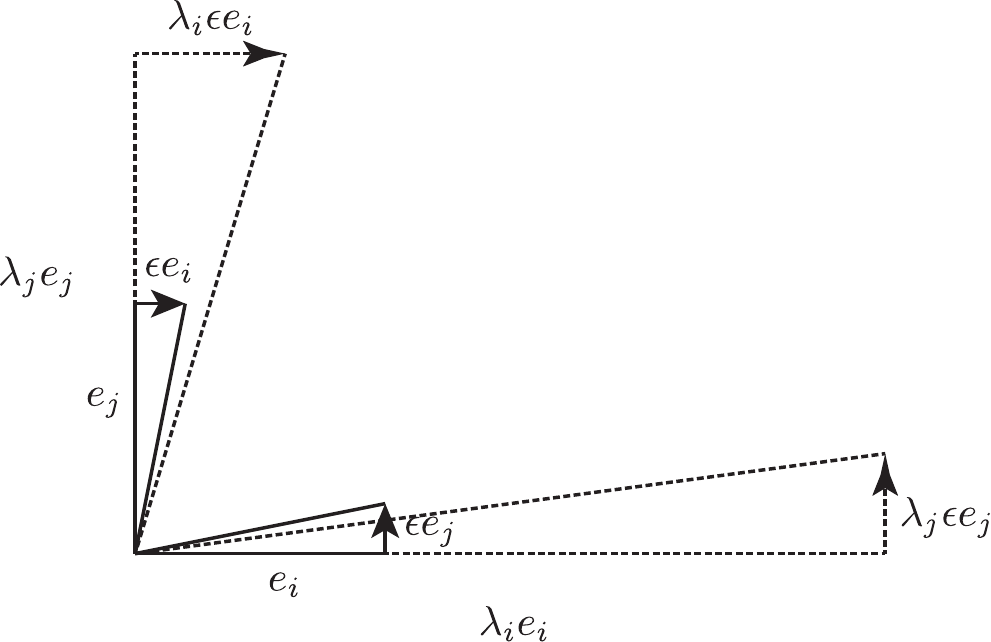}
\caption[Rotation of a perturbation as produced by the dynamics]{Rotation of a perturbation as produced by the dynamics. The vectors $e_i$ and $e_j$ are perturbed by $\e e_j$ and $\e e_i$, respectively, then allowed to flow along $\g$. The resulting vectors and perturbations are shown in dashed lines (here, $\la_i > \la_j$, and the rotation is toward $e_i$).}
\end{figure}

\nin {\bf Perturbation IV.}  For $1 \leq i < j \leq n$, let $$\tU_{IV}^{ij}(t) = U_A(t) + \psi^{ij}_{3}(t) - \psi^{ij}_{3}(t-(d-\de)).$$
Then, using Lemma \ref{III},
\begin{align*}
\De_{IV}^{ij} DP(d) &= DP(d-\de) \De DP(\de) - \De DP(\de) DP(d-\de) - \De DP(\de) DP(d-2\de) \De DP(\de) \\
                            &= DP(d-\de) \De DP(\de) - \De DP(\de) DP(d-\de) + O(\e^2\de^4) \\
                            &= \left( \left[ \begin{array}{c c} \Id & 0 \\ 0 & \Id \end{array} \right] + \left[ \begin{array}{c c} 0 & \Id \\ I(\la) & 0 \end{array} \right] d + O(d^2) \right) \De DP(\de) \\
                            & \hspace{1in} - \De DP(\de) \left( \left[ \begin{array}{c c} \Id & 0 \\ 0 & \Id \end{array} \right] + \left[ \begin{array}{c c} 0 & \Id \\ I(\la) & 0 \end{array} \right] d + O(d^2) \right) + O(\e^2\de^4) \\
                            &= d \left[ \begin{array}{c c} 0 & \Id \\ I(\la) & 0 \end{array} \right] \De DP(\de) - d \cdot \De DP(\de) \left[ \begin{array}{c c} 0 & \Id \\ I(\la) & 0 \end{array} \right] + O(\e\de^3 d^2) \\
                            &= d \left[ \begin{array}{c c} \De A' - \De B I(\la) & \De B' - \De A \\ (\De A - \De B') I(\la) & * \end{array} \right] + O(\e\de^3 d^2) \\
                            &= \left[ \begin{array}{c c} -d \cdot \int \Psi^{ij} I(\la) & 0 \\ 0 & * \end{array} \right] + O(\e\de^3 d^2),
\end{align*}
where the $*$ is an entry of $O(\e\de^3 d)$.  To write down $\int \Psi^{ij} I(\la)$, we can reduce to the $2 \times 2$ minor $\left[ \begin{array}{c c} a_{ii} & a_{ij} \\ a_{ji} & a_{jj} \end{array} \right]$, since all other entries are of higher order.  Then the $A$ component of the above matrix is
\begin{align*}
d\cdot \int \Psi^{ij} I(\la) &= \left[ \begin{array}{c c} 0 & \e\de^3 \\ \e\de^3 & 0 \end{array} \right] \cdot \left[ \begin{array}{c c} \la_i & 0 \\ 0 & \la_j \end{array} \right] \\
                                         &= \left[ \begin{array}{c c} 0 & \la_j \e\de^3 d \\ \la_i \e\de^3 d & 0 \end{array} \right],
\end{align*}
which is not symmetric when $\la_i \not= \la_j$ and can be decomposed into symmetric and anti-symmetric parts as
\begin{align*}
d\cdot \int \Psi^{ij} I(\la) &= \sS^{ij}(\la) + \sA^{ij}(\la) \\
                                          &= \e\de^3 d \left[ \begin{array}{c c} 0 & \frac{1}{2}(\la_i + \la_j) \\ \frac{1}{2}(\la_i + \la_j) & 0 \end{array} \right] 
+ \e\de^3 d \left[ \begin{array}{c c} 0 & \frac{1}{2}(\la_j -\la_i) \\ \frac{1}{2}(\la_i - \la_j) & 0 \end{array} \right].
\end{align*}

%%%%%%%%%%%%%%%%%%%%%
%  Inverse function theorem
%%%%%%%%%%%%%%%%%%%%%
\subsubsection{An open ball in $Sp(n)$}
Writing an element of $Sp(n)$ as 
$$\left[ \begin{array}{c c} A_{n\times n} & B_{n\times n} \\ A'_{n\times n} & B'_{n\times n} \end{array} \right],$$
consider the following coordinates on $Sp(n)$:
\begin{center}
\begin{tabular}{c l}
$\{ a'_{ij} + a'_{ji} \}$            & for $1 \leq i \leq j \leq n$ \\
$\{ a_{ij} + a_{ji} \}$             & for $1 \leq i \leq j \leq n$ \\
$\{ b_{ij} + b_{ji} \}$              & for $1 \leq i \leq j \leq n$ \\
$\{ a_{ij} - a_{ji} \}$               & for $1 \leq i < j \leq n$. 
\end{tabular}
\end{center}
Let $\sR(\g)$ be the space of curvature matrices along $\g$. Then $s \mapsto R + s \De R^{ij}_X$ gives a curve in $\sR(\g)$ through $R$ for each of the curvatures $\De R^{ij}_X$ ($X \in \{ I, II, III, IV \}$) produced in the above perturbations. These define a $(2n^2+n)$--dimensional subspace $S \subset \sR(\g)$. Consider the map $\Phi: \sR(\g) \rightarrow Sp(n)$ that takes a curvature matrix along $\g$ and returns the linear Poincar\'e map along $\g$ with the given curvature.  Using the above calculations, its derivative is given by
$$D\Phi |_{S} R = \left[ \begin{array}{c c c c} 2\psi_1^{ij} & 0 & 0 & 0 \\ 0 & 2\Psi_2^{ij} & 0 & 2\sS^{ij}(\la) \\ 0 & 0 & -4 \int \Psi_3^{ij} & 0 \\ 0 & 0 & 0 & 2\sA^{ij}(\la) \end{array} \right] + O(\e\de^3 d^2),$$
which for $\de$ small enough has full rank and therefore, by the Inverse Function Theorem, $\Phi |_S$ is a diffeomorphism.  In particlar, the image of a neighborhood of $R$ under $\Phi |_S$ contains a ball of radius $\de>0$ about $DP=DP(\g,g)$ in $Sp(n)$.  Since all constants in the above calculations depend only on the original metric $g$, the value of $H(g)$, and the size of $\De R$ (determined by the neighborhood $\sU$), $\de$ depends only on $g$ and $\sU$ (and is uniform over the geodesic $\g$).

This shows that we can perturb $DP(\g,g,t_0,t_0+d)$ in a ball of uniform size.  Since 
$$DP(\g,g,0,1) = DP(\g,g,t_0+d,1) \cdot DP(\g,g,t_0,t_0+d) \cdot DP(\g,g,0,t_0)$$
and the size of $DP(\g,g,a,b) \in Sp(n)$ is uniformly bounded above and below for $[a,b] \subset [0,1]$, this also shows that we can perturb $DP(\g,g,0,1)$ in a ball of uniform size.

\subsection{Perturbing $R$ by perturbing $g$}
In this section, for any one-parameter family of curvature matrices $\tR(t) = R(t) + \De R(t)$ that are $C^0$-close to $R(t)$, we define a metric $\tg$ supported in a tubular neighborhood $W$ of $\g$, show that it has Jacobi curvature matrix $\tR$ along $\g$, show that $\tg$ is $C^2$ close to $g$, and that this distance is independent of $W$.  Recall that in Fermi coordinates, $$\frac{\del}{\del x^i} \frac{\del}{\del x^j} g_{00}(t; 0) = - 2 R_{iooj}(t;0),$$ and that $R_{iooj}(t;0)$ are the components $R_{ij}(t)$ of the Jacobi curvature matrix (\cite{Kl78}).  Define a new metric $g^1$ in these coordinates by
$$g^1_{ij}(t;x) = \begin{cases} g_{00}(t;x) -2 \De R_{kl}(t) x^k x^l & \text{if} \, i=j=0 \\ g_{ij}(t;x) & \text{otherwise}. \end{cases}$$
Let $\vphi: \R^n \rightarrow \R$ be a $C^2$ bump function such that 
$$\vphi(x)=\begin{cases} 1 & \Norm{x} < 1/4 \\ 0 & \Norm{x} > 1, \end{cases}$$
and set $\vphi_\eta(x) = \vphi (x/\eta)$.  For the tubular neighborhood  $W = [0,1] \times (-\eta,\eta) \subset U$, define a new metric $\tg$ with $\supp (\tg - g) \subset W$ by
$$\tg(t;x) = \vphi_\eta(x) g^1(t;x) + (1-\vphi_\eta(x)) g(t;x).$$
Then
\begin{center}
\begin{tabular}{l l}
$\norm{\vphi_\eta(x)}_{C^0} = \norm{\vphi(x)}_{C^0}$                      & $\norm{\De R_{kl} x^k x^l}_{C^0, W} = \eta^2 \norm{\De R}$ \\
$\norm{\vphi_\eta(x)}_{C^1} \leq \eta^{-1}\norm{\vphi(x)}_{C^1}$   & $\norm{\De R_{kl} x^k x^l}_{C^1, W} = \eta \norm{\De R}$ \\
$\norm{\vphi_\eta(x)}_{C^2} \leq \eta^{-2}\norm{\vphi(x)}_{C^2}$   & $\norm{\De R_{kl} x^k x^l}_{C^2, W} = \norm{\De R}$,
\end{tabular}
\end{center}
so that
\begin{align*}
\Norm{\tg-g}_{C^2} &= \Norm{\vphi_\eta 2 \De R_{kl}x^k x^l}_{C^2} \\
                                   &\leq \Norm{\vphi_\eta(x)}_{C^0} \Norm{2 \De R_{kl}x^k x^l}_{C^2, W} + 2 \Norm{\vphi_\eta(x)}_{C^1} \Norm{2 \De R_{kl}x^k x^l}_{C^1, W} \\
                                   & \hspace{1cm} + \Norm{\vphi_\eta(x)}_{C^2} \Norm{2 \De R_{kl}x^k x^l}_{C^0, W} \\
                                   &\leq 8 \Norm{\vphi}_{C^2} \Norm{\De R}_{C^0},
\end{align*}
which does not depend on $\eta$.  Hence $\Norm{\tg-g}_{C^2} \leq O(\e)$ so that $\tg$ is a $C^2$-small perturbation of $g$.

The argument for avoiding perturbing the metric around a finite number of transverse geodesics follows the same lines as in Section \ref{FLGF-SS}.

\section*{Acknowledgements}
The author is grateful to Amie Wilkinson and Keith Burns for many valuable conversations, and also thanks Charles Pugh for useful discussions regarding this work.

\bigskip
{\small \textit{E-mail address:} \tt{davissch@umich.edu}} 


\begin{thebibliography}{}

\bibitem{ALD13} H.N. Alishah, J. Lopes Diaz, Realization of tangent perturbations in discrete and continuous time conservative systems. \emph{Preprint}, arXiv:1310.1063 (2013).

\bibitem{Arn02} M.-C. Arnaud, The generic symplectic $C^1$-diffeomorphisms of four-dimensional symplectic manifolds are hyperbolic, partially hyperbolic or have a completely elliptic periodic point.  \emph{Ergod. Th. \& Dynam. Sys.} {\bf 22}, 1621--1639 (2002).

\bibitem{BR08} M. Bessa, J. Rocha, On $C^1$-robust transitivity of volume-preserving flows.  \emph{J. Diff. Equations} {\bf 245}, 3127--3143 (2008).

\bibitem{BDP03} C. Bonatti, L. Diaz, E. Pujals, A $C^1$-generic dichotomy for diffeomorphisms: Weak forms of hyperbolicity or infinitely many sinks or sources.  \emph{Ann. of Math.} {\bf 158}, 355--418 (2003).

\bibitem{BGV06} C. Bonatti, N. Gourmelon, T. Vivier, Perturbations of the derivative along periodic orbits.  \emph{Ergod. Th. \& Dynam. Sys.} {\bf 26}, 1307--1337 (2006).

\bibitem{Con10} G. Contreras, Geodesic flows with positive topological entropy, twist maps and hyperbolicity. \emph{Ann. of Math.} {\bf 172}, 761--808 (2010).

\bibitem{CP02} G. Contreras, G. Paternain, Genericity of geodesic flows with positive topological entropy on $S^2$. \emph{J. Diff. Geom.} {\bf 61}, 1--49 (2002).

\bibitem{Es77} J-H. Eschenburg, Horospheres and the stable part of the geodesic flow. \emph{Math. Zeitschrift} {\bf 153}, 237--252 (1977).

\bibitem{Fr71} J. Franks, Necessary conditions for the stability of diffeomorphisms. \emph{Trans. A.M.S.} {\bf 158}, 301--308 (1971).

\bibitem{HT06} V. Horita, A. Tahzibi, Partial hyperbolicity for symplectic diffeomorphisms.  \emph{Ann. I.H. Poicar\'e} {\bf 23}, 641--661 (2006).

\bibitem{Kl78} W. Klingenberg, Lectures on Closed Geodesics.  Grundleheren Math. Wiss. 230, Springer--Verlag, New York, (1978).

\bibitem{Klo83} F. Klok, Generic singularities of the exponential map on Riemannian manifolds. \emph{Geom. Dedicata} {\bf 14}, 317--342 (1983).

\bibitem{MPP04} C. Morales, M.J. Pacifico, E. Pujals, Robust transitive singular sets for $3$-flows are partially hyperbolic attractors or repellers.  \emph{Ann. of Math.} {\bf 160}, 375--432 (2004).

\bibitem{Pat99} G. Paternain, Geodesic Flows.  Progress in Math. Vol. 180, Birkh\"auser (1999).

\bibitem{Viv05} T. Vivier, Robustly transitive $3$-dimensional regular energy surfaces are Anosov.  \emph{Institut de Math\'ematiques de Bourgogne, Dijon} Preprint 412 (2005).  \textsf{http://math.u-bourgogne.fr/topo/prepub/pre05.html}

\end{thebibliography}
\end{document}